\documentclass[a4paper,11pt,reqno]{amsart}
\usepackage{amsmath}
\usepackage{amsthm,enumerate}
\usepackage{mathtools}
\usepackage[normalem]{ulem}
\usepackage[T1]{fontenc}
\usepackage[utf8]{inputenc}
\usepackage[thinc]{esdiff}
\usepackage{amssymb}
\usepackage{color}
\usepackage{dsfont}
\usepackage{fancyhdr}
\usepackage{calc}
\usepackage{url}

\usepackage[text={6.0in,8.6in},centering]{geometry}

\usepackage[citecolor=blue,colorlinks]{hyperref}

\usepackage{xcolor}
\usepackage{xspace}
\newcommand{\R}{\mathbb{R}}

\newcommand{\supp}{\text{\rm supp}}

\newcommand{\Lip}{\mathrm{Lip}}

\newcommand{\diam}{{\rm{diam\,}}}

\newcommand{\RCD}{\mathsf{RCD}}
\newcommand{\CD}{\mathsf{CD}}

\newcommand{\Ent}{{\rm Ent}}

\DeclareMathOperator\arctanh{arctanh}

\newcommand{\norm}[1]{\left\Vert#1\right\Vert}

\newcommand{\abs}[1]{\left\vert#1\right\vert}
\newcommand{\Real}{\mathbb{R}}

\newcommand{\mm}{\mathfrak m}
\newcommand{\di}{\mathsf{d}}

\newcommand{\Ch}{\mathsf{Ch}}
\newcommand{\Per}{\mathsf{Per}}

\newcommand{\hed}{\mathsf{He}}

\DeclareMathOperator{\hk}{H\HKkern K}
\newcommand{\HKkern}{%
  \mkern-6.0mu
  \mathchoice{}{}{\mkern0.2mu}{\mkern0.5mu}%
}

\theoremstyle{plain}
\newtheorem{lemma}{Lemma}[section]
\newtheorem{theorem}[lemma]{Theorem}

\newtheorem{proposition}[lemma]{Proposition}
\newtheorem{corollary}[lemma]{Corollary}

\newtheorem*{theorem*}{Theorem}
\newtheorem*{maintheorem*}{Main Theorem}

\theoremstyle{definition}

\newtheorem{definition}[lemma]{Definition}
\newtheorem*{definition*}{Definition}
\newtheorem*{remark*}{Remark}
\newtheorem{remark}[lemma]{Remark}

\begin{document}
	\title[Lower bounds on Wasserstein distances]{
		Indeterminacy estimates, eigenfunctions and lower bounds on Wasserstein distances
	}
	\author{Nicol\`o De Ponti}\thanks{Nicol\`o De Ponti: Scuola Internazionale Superiore di Studi Avanzati (SISSA), Trieste,  Italy, \\email: ndeponti@sissa.it} 
	\author{Sara Farinelli}\thanks{Sara Farinelli: Scuola Internazionale Superiore di Studi Avanzati (SISSA), Trieste,  Italy, \\email: sfarinel@sissa.it}

	\begin{abstract}In the paper we prove two inequalities in the setting of $\RCD(K,\infty)$ spaces using similar  techniques. The first one is an indeterminacy estimate involving the $p$-Wasserstein distance between the positive part and the negative part of an $L^{\infty}$ function and the measure of the interface between the positive part and the negative part. The second one is a conjectured lower bound on the $p$-Wasserstein distance between the positive and negative parts of a Laplace eigenfunction. 
	\end{abstract}
	\maketitle
	
	\section{Introduction}
	In recent years, a growing interest has been devoted to the study of Wasserstein distances between positive and negative parts of a function, particularly in relation with indeterminacy estimates \cite{Steiner,SagSte,Steiner2,CaMaOr,CaFa,Mu,DuSa}. Given a closed (i.e. compact, without boundary), smooth, $n$-dimensional Riemannian manifold and denoting by $\mm$ its volume measure, one considers a nice enough function $f$ with zero mean and notices that if it is cheap to transport $f^{+}\mm$ to $f^{-}\mm$, then most of the mass of $f^{+}$ has to be close to most of the mass of $f^{-}$ and hence the zero set has to be large. The uncertainty principle quantify this relation by providing bounds from below on the quantity
	\begin{equation}\label{eq:uncer}
	W_p(f^{+}\mm,f^{-}\mm)\mathcal{H}^{n-1}\left(\{x : f(x)=0\}\right).
	\end{equation}
	Here $W_p$ denotes the $p$-Wasserstein distance, $\mathcal{H}^{n-1}$ is the $(n-1)$-dimensional Hausdorff measure and $ f^{+},\, f^{-}$ are the positive and negative parts of $f$, respectively.
	\medskip
	
	When $f$ is a Laplace eigenfunction, it is an intriguing problem to understand whether a meaningful upper bound on \eqref{eq:uncer} also holds. Questions related to the geometry of eigenfunctions are of central interest for different areas of mathematics and estimates on the quantity \eqref{eq:uncer} togheter with  Steinerberger's conjecture \cite{Steiner} allow to get estimates on the measure of nodal sets, in the flavour of Yau's conjecture \cite{Yau}.
	\medskip
	
	Around 40 years ago, Yau conjectured that there exists a positive constant $C$, depending only on the manifold, such that every eigenfunction $f_{\lambda}$, of eigenvalue $\lambda$, satisfies
	\begin{equation}\label{eq: yau}
	C\sqrt{\lambda}\ge \mathcal{H}^{n-1}\left(\{x: f_{\lambda}(x)=0\}\right)\ge \frac{\sqrt{\lambda}}{C}.
	\end{equation}
	
	We refer to \cite{LugMal} for a review of results related to Yau's conjecture, here we limit to mention that the lower bound in \eqref{eq: yau} was proved by Logunov \cite{Lug2}, while the upper bound has been very recently established for open regular subset of $\Real^n$ by Logunov, Malinnikova, Nadirashvili and Nazarov \cite{LogMalinn}, {but it remains open for compact manifolds (see \cite{Lug1} for a polynomial upper bound)}. 
	\medskip
	
	Regarding Wasserstein distances, Steinerberger proposed the following conjecture: for any $p\ge 1$ there exists a constant $C$, depending only on $p$ and on the manifold, such that for every non-constant eigenfunction $f_{\lambda}$, of eigenvalue $\lambda$, it holds
	\begin{equation}\label{eq: conj Stein}
	\frac{C}{\sqrt{\lambda}}\|f_{\lambda}\|_{L^{1}}^{\frac{1}{p}} \ge W_p( f^+_{\lambda}\mm, f^-_{\lambda}\mm) \ge \frac{1}{C\sqrt{\lambda}}\|f_{\lambda}\|_{L^{1}}^{\frac{1}{p}}\, .
	\end{equation}
	
	Some results are known in the direction of the upper bound. 
	In particular the first inequality in \eqref{eq: conj Stein}  with the non optimal factor $\sqrt{\log{\lambda}/\lambda}$ in place of $1/\sqrt{\lambda}$ was established by Steinerberger already in \cite{Steiner}. For $p=1$, the same non optimal upper bound was then extended to a more general class of spaces, the so called $\mathsf{RCD}(K,N)$ spaces, by Cavalletti and Farinelli \cite{CaFa}. 
	The sharp upper bound is known to hold for closed Riemannian manifolds and $p=1$ thanks to a recent result of Carroll, Massaneda and Ortega-Cerd\'a \cite{CaMaOr}.
	
	Concerning the lower bound no results were known when we firstly elaborated this note (see at the end of subsection \ref{subsec:lowerbound} for more details). 
	\medskip 
	
	It is worth to notice that a lower bound on the Wasserstein distance between $f^{+}_{\lambda}\mm$ and $f^{-}_{\lambda}\mm$ can be used to derive an upper bound on the measure of the nodal set of $f_{\lambda}$, provided an estimate from above of the quantity \eqref{eq:uncer} is established.
	
	\medskip
	The setting of this note will be the one of $\RCD(K,\infty)$ spaces. Roughly speaking, an $\mathsf{RCD}(K,\infty)$ space is a (possibly non-smooth) metric measure space having Ricci curvature bounded below by $K\in \R$ and no upper bound on the dimension, in a synthetic sense. We refer the reader to subsection \ref{subsec:curvaturedimension} for the precise definition, and here we only mention that the class of $\RCD(K,\infty)$ spaces was introduced in \cite{AmbrosioGigliSavare2} and includes: weighted Riemannian manifolds with Bakry-\'Emery Ricci curvature bounded below \cite{Sturm}, pmGH-limits of Riemannian manifolds with Ricci curvature bounded below \cite{GigliMondinoSavare}, finite dimensional Alexandrov spaces \cite{Petr}. In particular, every closed Riemannian manifold endowed with the geodesic distance and the volume measure is an $\RCD(K,\infty)$ space for some $K\in \R$.
	
	\medskip
	
	The present paper has two aims: the first one is to prove an indeterminacy estimate involving the $p$-Wasserstein distance between the positive part and the negative part of a general $L^{\infty}$ function with zero integral. The second one is to show that the lower bound in Steinerberger's conjecture holds in full generality.  Both the results will be established in spaces satisfying the $\RCD(K,\infty)$ condition and with an explicit computation of the constants appearing in the inequalities. The techniques that we use for proving the two inequalities are analogous and we remark that in the statements of our main results, namely Theorem \ref{theorem:main ind} and Theorem \ref{theorem:main eig}, we focus on the case $p=1$ since from this we easily derive the case $p>1$ with an argument explained in the proof of Corollary \ref{cor:main ind_p}. \\
	In the two following subsections we introduce and present respectively the two results. 
\subsection{Indeterminacy estimate}\label{subsec:indeterminacy}
	The topic of indeterminacy estimates involving the Wasserstein distance between the positive and negative parts of an $L^{\infty}$ function with zero integral and the measure of its zero set, of the type 
	\begin{equation*}
	W_1(f^{+}\mm,f^{-}\mm)\mathcal{H}^{n-1}\left(\{x : f(x)=0\}\right)\geq \left(\frac{\|f\|_{L^1}}{\|f\|_{L^{\infty}}}\right)^{\alpha}\|f\|_{L^1{}}\qquad \alpha>0,
	\end{equation*} 
 was firstly introduced by Steinerberger in \cite{Steiner} and \cite{Steiner2} in $1$ and $2$-dimensional spaces and then developed by Sagiv and Steinerberger for Euclidean domains in any dimension $n$ in \cite{SagSte}, with exponent $\alpha=4-\frac{1}{n}$. Then Carrol, Massaneda and Ortega-Cerd\'a \cite{CaMaOr} extended the estimate to general smooth, compact, Riemannian manifolds, also lowering (and thus improving) the exponent $ \alpha $. Finally, Cavalletti and the second author proved in \cite{CaFa} the estimate with the sharp exponent $\alpha=1$ in the even more general setting of metric measure spaces of finite diameter satisfying the so called curvature dimension condition $\CD(K,N)$, $K\in\Real$ and $N<+\infty$. We recall that in this context $K$ plays the role of a lower bound on the Ricci curvature and $N$ plays the role of an upper bound on the dimension. \medskip
	
	The first scope of this note is to prove the sharp (in the exponent) indeterminacy estimate in the possibly infinite dimensional setting of spaces satisfying the $\RCD(K,\infty)$ condition.
	
In order to properly state the result, we introduce the notation $h(X)$ for the Cheeger constant of the metric measure space $(X,\di,\mm)$ (see formula \eqref{eq:defChConst} for the definition of Cheeger constant). Here and below, we also tacitly assume to work with non-zero functions.
	\begin{theorem}\label{theorem:main ind}
		Let $(X,\di,\mm)$ be a space of finite measure satisfying the $\RCD(K,\infty)$ condition for some $K\in \Real$. Let $f\in  L^{\infty}(X,\mm)$ be such that $\int_{X}f\,d\mm=0$ and $\int_{X}\di(\bar{x},x)\abs{f(x)}\,d\mm(x)<+\infty$ for some $\bar{x}\in X$.  Then one has 	
		\begin{equation}\label{indeterminacy}
		W_1(f^+\mm,f^-\mm)\Per(\lbrace f>0 \rbrace)	\geq C(K,h(X)) \left(\frac{\|f\|_{L^{1}}}{\|f\|_{L^{\infty}}}\right)\|f\|_{L^{1}},
		\end{equation}
		with 
		\begin{equation*}
		C(K,h(X)):=\begin{dcases} \frac{\sqrt{\pi}}{27\sqrt{2}}	\qquad& K\geq 0 \, ,\\
		\bigg(1-\frac{1}{(2\pi)^{\frac{1}{4}}}\bigg)\frac{h(X)}{8h(X)+2\abs{K}^{\frac{1}{2}}}\qquad& K<0 \, . \end{dcases}
		\end{equation*}
	\end{theorem}

	The motivation behind our result is the observation that the estimate in \cite{CaFa} does not depend on $N$. This seems to suggest that an analogous result should hold in an “infinite dimensional setting”. The natural extension would have been to $\CD(K,\infty)$ spaces. The main tool used in \cite{CaFa} to prove the result in the $\CD(K,N)$ setting, namely the localization paradigm, however, is not at disposal in the $\CD(K,\infty)$ setting.  
	\\
	Our result  relies on very different techniques which are only available on the subclass of  $\CD(K,\infty)$ spaces which satisfy also the   $\RCD(K,\infty)$ condition.  In particular, the proof makes use of the heat flow and of its properties. The crucial ingredient is an inequality due to Luise and Savar\'e \cite[Theorem 5.2]{LuiseSavare}, linking the Wasserstein distance between two finite measures with the Hellinger distance of their evolution through the heat flow (see Proposition \ref{pr:hellinger-wass-hk} for the precise statement).
	\\
  Inequality  \eqref{indeterminacy} does not imply in general the indeterminacy estimate in \cite{CaFa}, being valid only in $\CD(K,\infty)$ spaces which satisfy the $\RCD(K,\infty)$ condition. However, it is worth to observe that at least in the $\RCD$ setting not only our result is more general than the one in \cite{CaFa}, since it can be applied to “infinite dimensional” spaces like Gaussian spaces, but also we do not require the space to have finite diameter (in opposite to \cite{CaFa}): \eqref{indeterminacy} is meaningful for the class of spaces having positive Cheeger constant, a class which includes spaces having finite diameter (see \cite{DeMoSe} for the details on this implication and for an example of space with finite measure, positive Cheeger constant and infinite diameter). 
  \medskip
  
  We remark that it is out of the purposes of this paper to find the optimal constant in the estimate. In this regard, we mention the recent work \cite{DuSa} where sharp (also in the constant) indeterminacy estimates have been established for spaces with a simple $1$-dimensional geometry.
	\medskip 
	
As we have already anticipated, the inequality \eqref{indeterminacy} can be easily extended to a sharp indeterminacy estimate for the $p$-Wasserstein distance, $p>1$, see Corollary \ref{cor:main ind_p}.\\
The proof of Theorem \ref{theorem:main ind} together with its Corollary is presented in Section \ref{indeterminacysection}, where we also prove a more refined result involving another transport distance, namely the Hellinger-Kantorovich distance \cite{LMS}. To avoid technicalities in the introduction, we refer to subsection \ref{sub:distances measures} for a brief presentation of this distance (see in particular the definition \eqref{eq: dyn HK}), and to Theorem \ref{th: ind hk} for the statement of the result.
	
	\subsection{Lower bound on the Wasserstein distance between eigenfunctions}\label{subsec:lowerbound}
	The second aim of this note is to prove the lower bound in the Steinerberger's conjecture \eqref{eq: conj Stein}. More precisely, we obtain:
	
	\begin{theorem}\label{theorem:manifold}
		Let $(\mathbb{M},g)$ be a smooth, closed, Riemannian manifold, and $p\ge 1$. Then there exists a constant $C(K,M,p)$ such that for any non-constant eigenfunction of the Laplacian $f_\lambda$, of eigenvalue $\lambda\geq M$, the following inequality is satisfied 	
		\begin{equation*}
		W_{p}(f_{\lambda}^{+}\mm,f_{\lambda}^{-}\mm)\geq   C(K,M,p)\frac{1}{\sqrt{\lambda}}\|f_{\lambda}\|_{L^{1}(\mathbb{M})}^{\frac{1}{p}}\, ,
		\end{equation*}
		with $K$ being a lower bound on the Ricci curvature of the manifold. 
	\end{theorem} 
	We remark that in the estimate there is no dependence on the dimension of the manifold. 
	From Theorem \ref{theorem:manifold} and the above mentioned upper bound obtained in \cite[Theorem 3]{CaMaOr} it follows exactly the full conjecture \eqref{eq: conj Stein} for $p=1$ and an equivalent formulation of Yau's conjecture:
	\begin{corollary}
		Let $(\mathbb{M},g)$ be a smooth, closed, Riemannian manifold. Then there exists a constant $C$, depending only on the manifold, such that for any  non-constant eigenfunction $f_\lambda$, of eigenvalue $\lambda$, the following inequality is satisfied 	
		\begin{equation*}
		\frac{C}{\sqrt{\lambda}}\|f_{\lambda}\|_{L^{1}(\mathbb{M})}\ge W_{1}(f_{\lambda}^{+}\mm,f_{\lambda}^{-}\mm)\ge \frac{1}{C\sqrt{\lambda}}\|f_{\lambda}\|_{L^{1}(\mathbb{M})}\, .
		\end{equation*}
		As a consequence, Yau's conjecture holds if and only if there exists a constant $C$, depending only on the manifold, such that for any eigenfunction $f_\lambda$ the following inequality is satisfied 	
		\begin{equation*}
		C\|f_{\lambda}\|_{L^{1}(\mathbb{M})}\ge W_{1}(f_{\lambda}^{+}\mm,f_{\lambda}^{-}\mm)\mathcal{H}^{n-1}\left(\{x: f_{\lambda}(x)=0\}\right)\ge \frac{\|f_{\lambda}\|_{L^{1}(\mathbb{M})}}{C}\, .
		\end{equation*}
	\end{corollary}
	\medskip
	
	As already emphasized, we obtain Theorem \ref{theorem:manifold} as an outcome of a more general result valid for a class of spaces which includes closed Riemannian manifolds. The case $p=1$ is stated in the following Theorem, while the case $p>1$ will be derived from the case $p=1$ and it is stated in Corollary \ref{cor:main eig}.
	\begin{theorem}\label{theorem:main eig}
		Let $M>0$, $K\in\Real$ and $(X,\di,\mm)$ be an $\RCD(K,\infty)$ space of finite measure. Then for any non-constant eigenfunction $f_{\lambda}$ of the Laplacian, of eigenvalue  $\lambda\ge M$ and satisfying $\int_{X}\di(\bar x, x)|f_{\lambda}(x)|\,d\mm(x)<+\infty$ for some $\bar x\in X$, it holds
		\begin{equation*}
		W_{1}(f_{\lambda}^{+}\mm,f_{\lambda}^{-}\mm)\geq   C(K,M)\frac{1}{\sqrt{\lambda}}\|f_{\lambda}\|_{L^{1}(X)}\, ,
		\end{equation*}
		where 
		\begin{equation}\label{constant:ck}
		C(K,M):=\begin{cases}\displaystyle e^{-\frac{1}{2}}\qquad&\text{if } K\geq 0\, ,\\ \displaystyle\left(1-\frac{K}{M}\right)^{\frac{M}{2K}-\frac{1}{2}}&\text{if } K<0\, .\end{cases}
		\end{equation}
	\end{theorem}

	\medskip
	
	Notice that in Theorem \ref{theorem:main eig} we are not requiring any compactness of the space $(X,\di)$, nor are we assuming that the spectrum of the metric measure space is discrete. The assumptions $\mm(X)<\infty$ and $\int_{X}\di(\bar x, x)|f_{\lambda}(x)|\,d\mm(x)<+\infty,$ trivially satisfied for compact spaces, are requested here to ensure that the measures $f^{+}_{\lambda}\mm,\,f^{-}_{\lambda}\mm$ have the same total mass and finite $1$-moment. 
	
	As for the indeterminacy estimate, the proof of Theorem \ref{theorem:main eig}, which will be given in Section \ref{sec: eigen}, relies on a crucial inequality that relates the Wasserstein distance between two finite measures with the Hellinger distance of their evolution through the heat flow (see Proposition \ref{pr:hellinger-wass-hk}). 
	\medskip 
	
	To conclude the introduction we notice that, up to our knowledge, the upper bound in \eqref{eq: conj Stein} conjectured by Steinerberger is open for any $p$ in $\RCD(K,\infty)$ spaces and for $p>1$ even in smooth Riemannian manifolds. Concerning the lower bound, only when the first version of this manuscript was in preparation we became aware of the concomitant work \cite{Mu}, where the author obtains the conjectured inequality for closed Riemmanian manifolds and $p=1$, with an implicit constant. It is interesting to notice that the proof in \cite{Mu} is based on ideas from elliptic PDEs and it makes use of a new, non-trivial, mass (non)-concentration property of Laplace eigenfunctions around their nodal set, while our approach does not require to appeal to any fine property satisfied by $f_{\lambda}$ (see Remark \ref{rmk: no properties}).

	\section*{Acknowledgements}
	The authors would like to thank Fabio Cavalletti for helpful discussions and many valuable suggestions.

\section{Preliminaries}\label{sec:prel}
In the sequel we will denote by $(X,\di)$ a complete and separable metric space. By $\mathcal{M}(X)$ we denote the space of  finite, non-negative, Borel measures on $X$.  We write $\mu\in \mathcal{M}_p(X)$  if $\mu\in \mathcal{M}(X)$ and there exists $\bar x\in X$ 
such that
\begin{equation*}
\int_{X}\di(\bar x,x)^p\, d\mu(x)<+\infty,
\end{equation*}
while $\mathcal{P}_p(X)\subset \mathcal{M}_p(X)$ denotes the subset of probability measures with finite $p$-moment. When $X$ is endowed with a Borel measure $\mm$, we denote by $L^p(X,\mm)$ the Lebesgue space of $p$-integrable (equivalence class of) functions, $p\in [1,\infty]$. For simplicity, we often write $L^p(X)$ (or $L^p$) in place of $L^p(X,\mm)$. 

We write $\mathcal{C}_{b}(X)$ to denote the space of real valued, bounded and continuous functions on $X$. The set of real valued (bounded, or with bounded support) Lipschitz functions is denoted by $\mathrm{Lip}(X)$ (respectively $\mathrm{Lip}_b(X)$ or $\mathrm{Lip}_{bs}(X)$). Finally, $\mathrm{B}_b(X)$ is the set of bounded Borel functions on $X$.

\subsection{Wasserstein, Hellinger-Kantorovich and Hellinger distances}\label{sub:distances measures}
\begin{definition}
Given $\mu_{1},\,\mu_{2}\in\mathcal{M}(X)$ and $p\in[1,+\infty)$, the $p$-Wasserstein distance $W_p$ between $\mu_1$ and $\mu_2$ is defined as 
	\begin{equation*}
	{W_p^p(\mu_1,\mu_2)}:=\inf\left \lbrace \int_{X\times X}\di(x,y)^{p}\,d\pi(x,y)\,\mid\pi\in\mathcal{M}(X\times X), \, (P_i)_{\sharp}\pi=\mu_i, i=1, 2\,\right \rbrace,
	\end{equation*}
	where $(P_i)_{\sharp}$ is the pushforward through the projection on the $i$-th component.
\end{definition}

Notice that $W_p(\mu_1,\mu_2)=+\infty$ whenever $\mu_1(X)\neq \mu_2(X)$, but $W_p(\mu_1,\mu_2)$ is finite if $\mu_{1},\,\mu_{2}\in\mathcal{M}_p(X)$ and have the same total mass. In particular, it is well known that $(\mathcal{P}_p(X),W_p)$ is a complete and separable metric space. The $p$-Wasserstein distance metrizes the weak convergence of measures plus convergence of the $p$-moment (see e.g. \cite{Villani}).

When $(X,\mathsf{d})$ is a length metric space, one can prove a dynamic formulation of the Wasserstein distance (see for instance \cite[Prooposition 2.10]{LuiseSavare}):
\begin{equation}\label{eq: dyn Was}
\frac{1}{p}W_p^p(\mu_0,\mu_1)=\sup\Big\{\int_X \zeta_1\,d\mu_1-\int_X \zeta_0\,d\mu_0, \ \zeta\in C^1([0,1],\mathrm{Lip}_b(X)),\, \partial_t\zeta_t+\frac{1}{q}|D\zeta_t|^q\le 0\Big\},
\end{equation}
where we are using the notation $\abs{D f}(x)$ for the slope of a Lipschitz function $f$ at the point $x$, i.e.
$$\abs{D f}(x):=\limsup_{y\to x} \frac{|f(y)-f(x)|}{\mathsf{d}(y,x)}.$$

We also introduce the weighted Hellinger-Kantorovich distance $\hk_{\alpha}$, $\alpha>0$, following the theory developed in \cite{LMS}. In its dynamical formulation on a length metric space $(X,\mathsf{d})$ it reads as follow
\begin{equation}\label{eq: dyn HK}
\hk^2_{\alpha}(\mu_0,\mu_1):=\sup\Big\{\int_X \zeta_1\,d\mu_1-\int_X \zeta_0\,d\mu_0, \ \zeta\in C^1([0,1],\mathrm{Lip}_b(X)),\, \partial_t\zeta_t+\frac{\alpha}{4}|D\zeta_t|^2+\zeta_t^2\le 0\Big\}.
\end{equation}
Notice that $\hk_{\alpha}(\mu_0,\mu_1)$ is finite even if $\mu_0(X)\neq \mu_1(X)$ and one can prove that $\hk_{\alpha}$ is indeed a distance on $\mathcal{M}(X)$. 

\medskip
\begin{definition}
	Given $\mu_{0},\,\mu_{1}\in \mathcal{M}(X)$ and $p\in[1,+\infty)$, the $p$-Hellinger distance $\hed_p$ (also called Matusita distance) \cite{Hellinger,Matusita} between $\mu_{0}$ and $\mu_{1}$ is defined as
	\begin{equation*}
	\hed^p_{p}(\mu_0,\mu_1):=\int_{X}\Big|\rho_{0}^{1/p}-\rho_{1}^{1/p}\Big|^{p}\,d\lambda,
	\end{equation*}
	where $\lambda$ is any dominating measure of $\mu_0,\,\mu_1$ and $\rho_i$ are the relative densities:  $\mu_{i}\ll\lambda$ and $\mu_{i}=\rho_i\lambda$ for $i=0,1$.
\end{definition}

We are particularly interested in the case $p=1$, which corresponds to the classical total variation, and $p=2$, which is the original distance studied by Hellinger. 	

An immediate consequence of the elementary inequality $|t-s| \ge \big|t^{1/2}-s^{1/2}\big|^{2}$ is that 
\begin{equation}\label{ineq: he1-he2}
\hed_1(\mu_0,\mu_1)\ge \hed^2_2(\mu_0,\mu_1) \qquad \textrm{for every} \ \mu_0,\mu_1\in \mathcal{M}(X).
\end{equation}

It is not difficult to show that all the $p$-Hellinger distances induce the same strong convergence of the total variation, and are complete distances on $\mathcal{M}(X)$.

It is useful to recall here that also the $p$-Hellinger distances admit a dynamic formulation for $p>1$ \cite[Proposition 2.8]{LuiseSavare}, specifically:
\begin{equation}\label{eq: dyn he}
\hed_p^p(\mu_0,\mu_1)=\sup\Big\{\int_X \zeta_1\,d\mu_1-\int_X \zeta_0\,d\mu_0, \ \zeta\in C^1([0,1],\mathrm{B}_b(X)),\, \partial_t\zeta_t+(p-1)\zeta_t^{\frac{p}{p-1}}\le 0\Big\}.
\end{equation}

\medskip
For every two measures $\mu_0,\mu_1\in \mathcal{M}(X)$ one can prove (see \cite[Chapter 7]{LMS}) the following relations between $\hed_2$, $W_2$ and $\hk$
\begin{align}
&\hk_{\alpha}(\mu_0,\mu_1)\le \hed_2(\mu_0,\mu_1) \qquad \textrm{and} \qquad \lim_{\alpha\downarrow 0} \hk_{\alpha}(\mu_0,\mu_1)=\hed_2(\mu_0,\mu_1), \\
\label{diseq: hk-W2}&\sqrt{\alpha}\hk_{\alpha}(\mu_0,\mu_1)\le W_2(\mu_0,\mu_1) \qquad \textrm{and} \qquad \lim_{\alpha\uparrow +\infty} \sqrt{\alpha}\hk_{\alpha}(\mu_0,\mu_1)=W_2(\mu_0,\mu_1).
\end{align}

\subsection{Metric measure spaces and curvature condition}\label{subsec:curvaturedimension}
In this section we recall some basic constructions in the theory of metric measure spaces, the definition of $\RCD(K,\infty)$ spaces and some of their properties, which will be useful later on. We refer to the survey \cite{AmbrosioICM} and the book \cite{GigliPasqualetto} as general references on the subject.

Our assumption on the space is that $(X,\di,\mm)$ is a metric measure space, briefly m.m.s., in the sense that $(X,\di)$ is a complete and separable metric space and $\mm$ is a non-negative, Borel measure defined on the Borel $\sigma$-algebra given by the metric $\di$. 

Although not needed for some of the results that we are going to discuss, we always assume in the paper $\mm(X)<\infty$ and $\supp(\mm)=X$. 	

Let $A\subset X$ be a Borel set, the perimeter $\mathrm{Per}(A)$ is defined as
\begin{equation*}
\mathrm{Per}(A):=\inf\bigg\{\liminf_{n\rightarrow \infty}\int_X |D f_n|\,d\mm: f_n\in \mathsf{Lip}_{bs}(X), f_n\rightarrow \chi_A \ \mathrm{in} \ L^1(X,\mm)\bigg\},
\end{equation*}
where we denote by $\chi_A:X\to\{0,1\}$ the indicator function of the set $A$.

The Cheeger constant of the metric measure space $(X,\mathsf{d},\mm)$ is defined as follows: 
\begin{equation}\label{eq:defChConst}
h(X):=\inf  \left\{\frac{\Per(A)}{\mm(A)}\, :\, A\subset X \textrm{ Borel with} \ 0<\mm(A)\leq \mm(X)/2 \right\}  .
\end{equation}

We define the relative entropy functional with respect to $\mm$, $\Ent_{\mm}:\mathcal{P}_2(X)\to [0,+\infty]$, as
\begin{align*}
\Ent_{\mm}(\mu):=\begin{cases}\int_{\lbrace\rho>0\rbrace}\rho\log(\rho)\,d\mm\qquad&\text{if } \mu=\rho\mm,\\+\infty\qquad&\text{otherwise}.\end{cases}
\end{align*}
In order to define the $\RCD(K,\infty)$ condition, we  need first to define the $\CD(K,\infty)$ introduced by Lott-Villani in \cite{LottVillani} and by Sturm in \cite{Sturm}.
\begin{definition}
We say that a m.m.s. $(X,\di,\mm)$ satisfies the $\CD(K,\infty)$ condition if for any couple of measures $\mu^0,\,\mu^1\in\mathcal{P}_2(X)$ with $\Ent_{\mm}(\mu_i)<+\infty$, $i=0,1$, there exists a $W_2$-geodesic $\{\mu_t\}_{t\in[0,1]}$ such that $\mu_0=\mu^0$, $\mu_1=\mu^1$ and for every  $t\,\in(0,1)$
\begin{equation*}
\Ent_{\mm}(\mu_t)\leq (1-t)\Ent_{\mm}(\mu_0)+t\Ent_{\mm}(\mu_1)-\frac{K}{2}t(1-t)W_2^{2}(\mu_0,\mu_1).
\end{equation*}
\end{definition}
For $f\in L^{2}(X,\mm)$ we define the Cheeger energy (see \cite{Cheeger}) as 
\begin{equation*}
\Ch(f): = \inf \left\{\liminf_{n\to \infty}\frac{1}{2}\int \abs{D f_n}^2 \, d\mm \colon 
f_n\in \Lip(X)\cap L^2(X,\mm),\,
\norm{f_n-f}_{L^2}\to 0
\right\},
\end{equation*}
and we put
$$W^{1,2}(X,\di,\mm):=\{f\in L^{2}(X,\mm)\,:
\, \Ch{(f)}<+\infty\}.$$

For simplicity, we will often drop the dependence of the metric measure structure and write $W^{1,2}(X)$ (or $W^{1,2}$) in place of $W^{1,2}(X,\di,\mm)$.

For any $f\in W^{1,2}(X)$, the Cheeger energy admits an integral representation 
$$\Ch(f)=\frac{1}{2}\int_{X}|Df|^2_{w}\,d\mm \, ,$$
where $|Df|_{w}$ is called minimal weak upper gradient.  
\begin{definition}\label{def:chquadr}
Following \cite{AmbrosioGigliSavare2}, we say that a m.m.s. $(X,\di,\mm)$ satisfies the $\RCD(K,\infty)$ condition if it satisfies the $\CD(K,\infty)$ condition and in addition the Cheeger energy $\Ch$ is a quadratic form on $W^{1,2}(X,\di,\mm)$, i.e. for every $f$ and $g$ $\in W^{1,2}(X,\di,\mm)$ the following equality is satisfied 
\begin{equation*}
\Ch(f+g)+\Ch(f-g)=2\Ch(f)+2\Ch(g).
\end{equation*}
\end{definition}
We remark that $\Ch$ is a convex and lower semicontinuous functional over $L^{2}(X,\mm)$. This implies that $W^{1,2}(X,\di,\mm)$ is a Banach space with the norm 
\begin{equation*}
\|f\|^2_{W^{1,2}(X)}:= \|f\|^2_{L^2(X)}+2{\Ch(f)},
\end{equation*}
which turns out to be an Hilbert space if $X$ satisfies the $\RCD(K,\infty)$ condition. 
\medskip\\
We focus from now on on $(X,\di,\mm)$ satisfying the $\RCD(K,\infty)$ condition. \\
It is useful  to recall the definition of subdifferential for $\Ch$. Given $f\in W^{1,2}(X)$, we say that $g\in \partial^{-}\Ch(f)$, namely $g$ is in the subdifferential of $\Ch$ at $f$, if 
\begin{equation*}
\int_{X}g(\psi-f)\,d\mm\leq \Ch(\psi)-\Ch(f) \qquad \forall \, \psi\,\in L^{2}(X).
\end{equation*}
In an $\RCD(K,\infty)$ space, the subdifferential of $\Ch$ where non empty is single valued. 
From the convexity and lower semicontinuity of $\Ch$ and from the fact that $W^{1,2}(X)$ is dense in $L^{2}(X)$, it  follows, using the theory of gradient flows in Hilbert spaces, that for any $f\in L^{2}(X)$ there exists a unique locally absolutely continuous curve $t\mapsto H_t(f)$, $t\in (0,+\infty)$, with values in $L^2(X)$, which satisfies
\begin{equation*}
\begin{cases}\frac{d}{dt}H_tf= -\partial^{-}{\Ch}(H_tf)\quad &\text{a.e.}\ t>0,\\\lim_{t\to 0}H_tf=f \quad &\text{in } L^{2}(X).\end{cases}
\end{equation*}
$\{H_{t}\}_{t\geq 0}$ is called the heat semigroup and for any $t>0$, $f\mapsto H_{t}f$ is a linear contraction in $L^{2}(X)$. By the density of $L^{2}(X)\cap L^{p}(X)$ in $L^{p}(X)$, it can be extended to a semigroup of linear contractions in any $L^{p}(X)$, $p\geq 1$. It can also be extended to $L^{\infty}(X)$ and it is known that $H_tf$, for $f\in L^{\infty}(X)$, admits an integral representation via the heat kernel.

We remark that in our setting, the heat semigroup satisfies the maximum principle: 
\begin{equation}\label{property:maxprinciple}
H_{t}f\leq C \quad \text{ if } f\leq C \quad \mm\text{-a.e.}\, ,
\end{equation}
from which follows that it is sign preserving. Moreover, $H_t$ is also measure preserving
\begin{align*}
\int_{X}H_tf\,d\mm=\int_{X}f\,d\mm, \quad \forall f \in L^1(X), \quad \forall\, t>0,
\end{align*}
and for any $f\in L^{\infty}(X,\mm)$ we have that $H_tf$ belongs to the space $\Lip_{b}(X)$, with the bound \cite[Proposition 3.1]{DepontiMondino}
\begin{equation}\label{strong Feller}
\begin{aligned}
&\| \,|D H_tf|_w\, \|_{L^{\infty}}\leq \sqrt{\frac{2K}{\pi(e^{2Kt}-1)}} \; \| f\|_{L^{\infty}}\, \quad \textrm{if} \ K\neq 0, \\
&\| \,|D H_tf|_w\, \|_{L^{\infty}}\leq \sqrt{\frac{ 1 }{\pi t}}\; \| f\|_{L^{\infty}} \quad \textrm{if} \ K=0,
\end{aligned}
\end{equation}
(which is sharp in the case $K>0$).

From properties \eqref{property:maxprinciple} and \eqref{strong Feller}, $H_t$ maps $\mathcal{C}_{b}(X)$ into itself, so it is defined its adjoint operator $H_t^{\ast}:\mathcal{P}(X)\to \mathcal{P}(X)$ that satisfies
\begin{equation}\label{eq: heatadj}
 H^{\ast}_t(\rho\mm)=H_t(\rho)\mm 
\end{equation}
for any probability density $\rho\in L^1_{+}(X,\mm)$ (see \cite[Proposition 3.2]{AmbrosioGigliSavare1} for details).

Finally, we recall that an $\mathsf{RCD}$ space is a length space, and thus formula \eqref{eq: dyn Was} and \eqref{eq: dyn HK} hold in this setting.

As shown by Luise and Savaré \cite{LuiseSavare}, the regularizing effect of the heat semigroup $H_t$ allows to control the stronger $p$-Hellinger distance in terms of the weaker $p$-Wasserstein and Hellinger-Kantorovich distances. 

To properly formulate their results, crucial for our purposes, first of all we set for $t>0$
\begin{equation}\label{def:rk}
R_{K}(t):=\begin{cases}\frac{e^{2Kt}-1}{K}\qquad&\text{if } K\neq 0, \\2t&\text{if } K= 0.\end{cases}
\end{equation} 

\begin{proposition}\cite[Theorem 5.2 and 5.4]{LuiseSavare}\label{pr:hellinger-wass-hk}
	Let $(X,\di,\mm)$ be an $\RCD(K,\infty)$ metric measure space for some $K\in\Real$, and let $p\in [1,2]$. For $\mu_0,\mu_1\in\mathcal{P}_p(X)$ it holds
	\begin{equation}\label{inequality:hellinger-wass}
	W_{p}(\mu_0,\mu_1)\ge p(R_K(t))^{\frac{1}{2}}\hed_{p}(H_t^{\ast}\mu_0,H_t^{\ast}\mu_1) \quad \forall\, t>0.
	\end{equation}
Moreover, for every $\mu_0,\mu_1\in \mathcal{M}(X)$ it holds 
\begin{equation}\label{inequality:hellinger-hk}
	\hk_{4R_K(t)}(\mu_0,\mu_1)\ge \hed_{2}(H_t^{\ast}\mu_0,H_t^{\ast}\mu_1) \quad \forall\, t>0.
\end{equation}

Here $R_K(t)$ is the function defined in \eqref{def:rk}.
\end{proposition}
Notice that the estimate \eqref{inequality:hellinger-hk} is more refined than \eqref{inequality:hellinger-wass} (as a consequence of \eqref{diseq: hk-W2}), at a cost of being more implicit in the sense that both the left hand side and the right hand side depend on $t$. 

To conclude the section, we recall in the following Proposition an intermediate result contained in the proof of the Buser's inequality given in \cite{DepontiMondino}, to which we refer for all the details. For the reader convenience, we give here a sketch of the proof. 

Let us define:
\begin{equation}\label{eq:ExplJK}
J_K(t):=\displaystyle\int_0^t \sqrt{\frac{2}{\pi R_K(s)}}\,ds=\begin{cases}\sqrt{\frac{2}{\pi K}}\arctan\Big(\sqrt{e^{2Kt}-1}\Big)  \ \ &\textrm{if} \ \  K>0,\\
\frac{2}{\sqrt{\pi}}\sqrt{t} \ \ &\textrm{if} \ \ K=0,\qquad t>0.\\ 
\sqrt{-\frac{2}{\pi K}}\arctanh{\Big(\sqrt{1-e^{2Kt}}\Big)} \ \ &\textrm{if} \ \ K<0,\end{cases}
\end{equation}

\begin{proposition}\label{prop:per}
Let $(X,\di,\mm)$ be a space of finite measure satisfying the $\RCD(K,\infty)$ condition for some $K\in \mathbb{R}$. Let $A\subseteq X$ be a Borel set. Then 
\begin{equation*}
\int_{A^c}{H_t(\chi_{A})}\,d\mm\leq \frac{1}{2}{J_K(t)}\Per(A),
\end{equation*}
where $J_K(t)$ was defined in \eqref{eq:ExplJK}.
\end{proposition}
\begin{proof}
By the above mentioned regularizing effect of the heat semigroup we know (\cite[Proposition 3.1]{DepontiMondino}) that for every function $f\in L^{\infty}(X)$ it holds
\begin{equation}\label{eq: infboundgrad}
\||D(H_tf)|_{w}\|_{L^{\infty}}\le \sqrt{\frac{2}{\pi R_K(t)}}\|f\|_{L^{\infty}}\,,
\end{equation} 
where $R_K(t)$ was defined in \eqref{def:rk}. By a duality argument, from \eqref{eq: infboundgrad} one easily derives 
\begin{equation}\label{eq: 1boundgrad}
\|f-H_t(f)\|_{L^1}\le J_K(t)\| |Df|_{w}\|_{L^1}\,, 
\end{equation}
say for $f\in \Lip_{bs}(X)$. Now, for any Borel set $A$ we consider a sequence $f_n\in \Lip_{bs}(X)$, $f_n\to \chi_A$ in $L^1(X)$, recovery sequence for $\Per(A)$. By applying \eqref{eq: 1boundgrad} to $f_n$ and passing to the limit $n\to \infty$ we deduce
\begin{align}
{J_K(t)}\Per(A)&\ge \| \chi_A-H_t(\chi_A)\|_{L^1}=\int_A [1-H_t(\chi_A)]d\mm + \int_{A^c}H_t(\chi_A)d\mm \\
&= \int_X [1-H_t(\chi_A)]d\mm-\int_{A^c}1\,d\mm+2\int_{A^c}H_t(\chi_A)d\mm=2\int_{A^c}H_t(\chi_A)d\mm\,.
\end{align}
as desired.
\end{proof}

\subsubsection*{Laplacian and Eigenfunctions}
From the Cheeger energy arises also the definition of Laplacian:
\begin{definition}
Let $(X,\di,\mm)$ satisfying the $\RCD(K,\infty)$ condition for some $K\in \mathbb{R}$. For  $f$ in $W^{1,2}(X)$, the Laplacian of $f$ is defined as  $\Delta f:=-\partial^- \Ch (f)$, provided that $\partial^- \Ch (f)$ is non empty. 
\end{definition}
We say that a non-zero function $f_{\lambda}\in W^{1,2}(X)$ is an eigenfunction of the Laplacian  of eigenvalue $\lambda \in [0,+\infty)$ if $-\Delta f_{\lambda}=\lambda f_{\lambda}$.
If one considers the evolution at time $t$  via the heat flow of an eigenfunction $f_\lambda$, then 
\begin{equation*}
H_tf_\lambda=e^{-\lambda t}f_\lambda. 
\end{equation*}

Every non-zero constant function is an eigenfunction of eigenvalue $0$ (recall that we are assuming $\mm(X)<+\infty$), and every other eigenfunction has zero mean, meaning that
\begin{equation*}
\int_{X}f_\lambda^-\,d\mm=\int_{X}f_\lambda^+\,d\mm.
\end{equation*}

Under our quite general assumptions, the spectrum of the Laplacian may not be discrete. For brevity, we refer the reader to \cite[Proposition 6.7]{GigliMondinoSavare} and \cite[Theorem 2.17]{DeMoSe} for some results about the spectrum of the Laplacian on $\RCD(K,\infty)$ spaces.
Here we just mention that the condition $\diam(X)<\infty$, or $K>0$, implies the compactness of the embedding of $W^{1,2}(X)$ into $L^2(X)$, and thus the existence of a basis of $L^2(X)$ formed by eigenfunctions corresponding to a diverging sequence of eigenvalues.

\section{Indeterminacy estimate}\label{indeterminacysection}
We start by proving a Proposition which is linked to Proposition \ref{prop:per}.

\begin{proposition}\label{pr:sqrt heat}
Let $(X,\di,\mm)$ be a metric measure space of finite measure satisfying the $\RCD(K,\infty)$ condition for some $K\in \mathbb{R}$, and let $f\in L^{\infty}(X,\mm)$. Then
\begin{equation*}
\int_{X}\sqrt{H_t(f^+)H_t(f^{-})}\,d\mm\leq J_{K}(t)^{\frac{1}{2}}\Per(\lbrace x\in X \, |\, f(x)>0\rbrace)^{\frac{1}{2}}\|f\|^{\frac{1}{2}}_{L^{1}}\|f\|^{\frac{1}{2}}_{L^{\infty}},
\end{equation*}
where $J_K(t)$ was defined in \eqref{eq:ExplJK}.
\end{proposition}
\begin{proof}
By taking advantage of the maximum principle for the heat semigroup, the Cauchy-Schwarz inequality and Proposition \ref{prop:per}, one has
\begin{align*}
&\int_{\lbrace f>0\rbrace}\sqrt{H_t(f^+)H_t(f^{-})}\,d\mm\leq \|f^-\|^{\frac{1}{2}}_{L^{\infty}}\int_{\lbrace f>0\rbrace}\sqrt{H_t(f^+)H_t(\chi_{\lbrace f\leq 0\rbrace})}\,d\mm \\
&\leq \|f^-\|^{\frac{1}{2}}_{L^{\infty}}\|H_{t}(f^+)\|^{\frac{1}{2}}_{L^1}\left(\int_{\lbrace f>0\rbrace}{H_t(\chi_{\lbrace f\leq 0\rbrace})}\,d\mm\right)^{\frac{1}{2}}\\
&\leq\frac{1}{\sqrt2}\|f^-\|^{\frac{1}{2}}_{L^{\infty}}\|f^+\|^{\frac{1}{2}}_{L^1}{J_K(t)}^{\frac{1}{2}}\Per(\lbrace f>0\rbrace)^{\frac{1}{2}},
\end{align*}
where we have also used that the heat flow is mass preserving. Along the same lines, one also gets
$$\int_{\lbrace f\le0\rbrace}\sqrt{H_t(f^+)H_t(f^{-})}\,d\mm\leq\frac{1}{\sqrt2}\|f^{+}\|^{\frac{1}{2}}_{L^{\infty}}\|f^-\|^{\frac{1}{2}}_{L^1}{J_K(t)}^{\frac{1}{2}}\Per(\lbrace f\leq 0\rbrace)^{\frac{1}{2}}.$$
In particular splitting the integral in the statement of the proposition in an integral on the set where $f$ is positive, and an integral on the set where $f$ is non-negative, we deduce
\begin{align*}
&\int_{X}\sqrt{H_t(f^+)H_t(f^{-})}\,d\mm=\int_{\lbrace f>0\rbrace}\sqrt{H_t(f^+)H_t(f^{-})}\,d\mm+\int_{\lbrace f\leq 0\rbrace}\sqrt{H_t(f^+)H_t(f^{-})}\,d\mm \\ \leq&
 \frac{1}{\sqrt2}\|f^{-}\|^{\frac{1}{2}}_{L^{\infty}}\|f^+\|^{\frac{1}{2}}_{L^1}{J_K(t)}^{\frac{1}{2}}\Per(\lbrace f>0\rbrace)^{\frac{1}{2}}+\frac{1}{\sqrt2}\|f^{+}\|^{\frac{1}{2}}_{L^{\infty}}\|f^-\|^{\frac{1}{2}}_{L^1}{J_K(t)}^{\frac{1}{2}}\Per(\lbrace f\leq 0\rbrace)^{\frac{1}{2}}.
\end{align*}
The conclusion follows by observing that $\Per(\lbrace f>0\rbrace)=\Per(\lbrace f\leq 0\rbrace)$, $\|f^\pm\|_{L^{\infty}}\leq \|f\|_{L^{\infty}}$ and $\|f^+\|_{L^1}+\|f^-\|_{L^1}=\|f\|_{L^1}$.
\end{proof}

In the course of the proof of Theorem \ref{theorem:main ind} we also take advantage of the following easy Lemma.
\begin{lemma}\label{lem: norm-cheeg}
Let $(X,\di,\mm)$ be a metric measure space of finite measure. Then, for every $f\in L^{\infty}(X,\mm)$ of null mean we have
\begin{equation}\label{eq: norm-cheeg}
\frac{\|f\|_{L^{\infty}}\Per(\lbrace f > 0\rbrace)}{\|f\|_{L^1}}\geq \frac{h(X)}{2},
\end{equation}
where $h(X)$ is the Cheeger constant of the space defined in \eqref{eq:defChConst}.
\end{lemma}
\begin{proof}
We can suppose without loss of generality that $\mm(\{f>0\})\leq \mm(X)/2$ (since the left hand side of \eqref{eq: norm-cheeg} does not change if we replace $f$ with $-f$). We have
$$\|f\|_{L^1}=\int_X f^{+}d\mm+\int_X f^{-}d\mm=2\int_{\{f>0\}}f^{+}d\mm\leq 2\mm(\{f>0\})\|f\|_{L^{\infty}}.$$
As a consequence,
$$\frac{\|f\|_{L^{\infty}}\Per(\lbrace f > 0\rbrace)}{\|f\|_{L^1}}\geq \frac{\Per(\lbrace f > 0\rbrace)}{2\mm(\{f>0\})}\ge \frac{h(X)}{2}$$
where in the last passage we have used the definition of Cheeger constant, since the set $\{f>0\}$ is a possible competitor in the right hand side of \eqref{eq:defChConst}.
\end{proof}

We are now able to prove the indeterminacy estimate.
\begin{proof}[Proof of Theorem \ref{theorem:main ind}]
We divide the proof in two steps.

\vspace{3mm}
\emph{Step $1$: general estimate involving time.}

Using the inequality \eqref{inequality:hellinger-wass} with $p=1$, the definition of $H_t^{\ast}$ \eqref{eq: heatadj} and the inequality  \eqref{ineq: he1-he2} we have that for every $t>0$
\begin{equation}\label{ineq:W1-he2 ind}
\begin{aligned}
&W_1(f^+\mm,f^-\mm)\geq  R_K(t)^{\frac{1}{2}}\hed_1\big(H_t(f^+)\mm,H_t(f^-)\mm\big)\\
&\geq R_K(t)^{\frac{1}{2}}\hed_2^2\big(H_t(f^+)\mm,H_t(f^-)\mm\big).
\end{aligned}
\end{equation}
Now we make use of the explicit expression of $\hed_2$, of the mass preservation property of the heat flow, and of Proposition \ref{pr:sqrt heat} to obtain
\begin{equation}\label{ineq:he2-per ind}
\begin{aligned}
&\hed_2^2(H_t(f^+)\mm,H_t(f^-)\mm)=\int_X \Big(H_t(f^+)+H_t(f^-)-2\sqrt{H_t(f^+)H_t(f^-)}\Big)\,d\mm\\
&\ge \|f\|_{L^1}-2J_{K}(t)^{\frac{1}{2}}\Per(\{f(x)>0\})^{\frac{1}{2}}\|f\|^{\frac{1}{2}}_{L^{1}}\|f\|^{\frac{1}{2}}_{L^{\infty}}\, .
\end{aligned}
\end{equation}

By putting together \eqref{ineq:W1-he2 ind} and \eqref{ineq:he2-per ind} we thus obtain that for every $t>0$
\begin{equation}\label{ineq: W_1-t ind}
W_1(f^+\mm,f^-\mm)\geq R_K(t)^{\frac{1}{2}} \|f\|_{L^1}-2\Big( R_K(t)J_{K}(t)\Per(\{f(x)>0\})\|f\|_{L^{1}}\|f\|_{L^{\infty}}\Big)^{\frac{1}{2}}.
\end{equation}

\vspace{5mm}
\emph{Step $2$: optimizing in $t$.}

In the case $K=0$ the right hand side of \eqref{ineq: W_1-t ind}, that we denote with $g(t)$, has the following expression
\begin{equation}
g(t)=\sqrt{2}\|f\|_{L^1}\,t^{\frac{1}{2}}-\frac{4}{\pi^{\frac{1}{4}}}\|f\|^{\frac{1}{2}}_{L^1}\|f\|^{\frac{1}{2}}_{L^{\infty}}\Per(\lbrace f> 0\rbrace)^{\frac{1}{2}\,}t^{\frac{3}{4}}.
\end{equation}
By choosing 
$$\bar{t}=\frac{\pi}{324}\frac{\|f\|_{L^1}^2}{\|f\|_{L^\infty}^2\Per(\lbrace f> 0\rbrace)^2}$$
we maximize the function $g$ and we obtain
\begin{align*}
W_1(f^+\mm,f^-\mm)\geq g(\bar{t})=&\bigg(\sqrt{2}\sqrt{\frac{\pi}{324}}-\frac{4}{\pi^{\frac{1}{4}}}\frac{\pi^{\frac{3}{4}}}{324^{\frac{3}{4}}}\bigg)\frac{\|f\|_{L^1}^2}{\|f\|_{L^\infty}\Per(\lbrace f> 0\rbrace)}\\
&=\frac{\sqrt{\pi}}{27\sqrt{2}}\frac{\|f\|_{L^1}^2}{\|f\|_{L^\infty}\Per(\lbrace f> 0\rbrace)}\, .
\end{align*}
\vspace{3mm}

For $K<0$ we use again the notation $g(t)$ for the right hand side of \eqref{ineq: W_1-t ind} so that
\begin{equation}
g(t)=D_K(f)\sqrt{1-e^{2Kt}}\bigg[1-\frac{2^{\frac{5}{4}}}{\pi^{\frac{1}{4}}}\Big(D_K(f)\arctanh(\sqrt{1-e^{2Kt}})\Big)^{\frac{1}{2}}\bigg]\frac{\|f\|^2_{L^1}}{\|f\|_{L^{\infty}}\Per(\lbrace f > 0\rbrace)},
\end{equation}
where we have denoted by $D_K(f)$ the quantity
$$D_K(f):=\frac{\|f\|_{L^{\infty}}\Per(\lbrace f > 0\rbrace)}{\|f\|_{L^1}|K|^{\frac{1}{2}}}.$$
We use the change of variable $(0,1)\ni s:=\sqrt{1-e^{2Kt}}$ and we consider the function
$$g_1(s):=D_K(f)s\bigg[1-\frac{2^{\frac{5}{4}}}{\pi^{\frac{1}{4}}}\big(D_k(f)\arctanh(s)\big)^{\frac{1}{2}}\bigg]\qquad s\in (0,1).$$ 
We recall now the elementary inequality 
$$\arctanh(s)\leq \frac{s}{1-s}\qquad s\in (0,1),$$
so that 
$$g_1(s)\geq D_K(f)s\bigg[1-\frac{2^{\frac{5}{4}}}{\pi^{\frac{1}{4}}}\Big(D_k(f)\frac{s}{1-s}\Big)^{\frac{1}{2}}\bigg]=:g_2(s)\qquad s\in (0,1).$$
We finally take the admissible choice 
$$\bar{s}:=\frac{1}{8D_K(f)+1}$$
and, putting everything together, we obtain
\begin{align}
\nonumber W_1(f^+\mm,f^-\mm)&\geq g_2(\bar{s})\frac{\|f\|^2_{L^1}}{\|f\|_{L^{\infty}}\Per(\lbrace f > 0\rbrace)}\\
\label{eq: final est W1 K<0} &=\bigg(1-\frac{1}{(2\pi)^{\frac{1}{4}}}\bigg)\frac{D_K(f)}{8D_K(f)+1} \frac{\|f\|^2_{L^1}}{\|f\|_{L^{\infty}}\Per(\lbrace f > 0\rbrace)}\, .
\end{align}
Notice that, thanks to Lemma \ref{lem: norm-cheeg} we know that
\begin{equation}\label{eq: use of lem norm-cheeg}
D_K(f)\geq h(X)/(2|K|^{\frac{1}{2}}).
\end{equation}
 Moreover, the function 
$$x\mapsto \frac{x}{8x+1} \qquad x>0,$$
is increasing, so that we can bound from below the right hand side of \eqref{eq: final est W1 K<0} using \eqref{eq: use of lem norm-cheeg} and obtain
$$W_1(f^+\mm,f^-\mm)\geq \bigg(1-\frac{1}{(2\pi)^{\frac{1}{4}}}\bigg)\frac{h(X)}{8h(X)+2|K|^{\frac{1}{2}}} \frac{\|f\|^2_{L^1}}{\|f\|_{L^{\infty}}\Per(\lbrace f > 0\rbrace)}\, ,$$
which concludes the proof.
\end{proof}

In the next corollary we show how to obtain an indeterminacy estimate for the $p$-Wasserstein distance as a simple consequence of the indeterminacy estimate for the $1$-Wasserstein distance. 
\begin{corollary}\label{cor:main ind_p}
Let $(X,\di,\mm)$ be a metric measure space of finite measure satisfying the $\RCD(K,\infty)$ condition for some $K\in \mathbb{R}$, and let $f\in L^{\infty}(X,\mm)$ with null mean and satisfying $\int_X \di(\bar x, x)|f_{\lambda}(x)|d\mm(x)<+\infty$ for some $\bar{x}\in X$. Then, for any $p>1$	
\begin{equation}\label{ineq:indp}
W_p(f^+\mm,f^-\mm)\Per(\lbrace f>0 \rbrace)	\geq 2^{\frac{p-1}{p}}C(h(X),K)\left(\frac{\|f\|_{L^{1}}}{\|f\|_{L^{\infty}}}\right)\|f\|^{\frac{1}{p}}_{L^{1}}\, ,
\end{equation}
where $C(h(X),K)$ is the constant appearing in Theorem \ref{theorem:main ind}.
\end{corollary}
\begin{proof}
The result follows from Theorem \ref{theorem:main ind} and the bound
\begin{equation}\label{eq:p-1}
W_p(f^+\mm,f^-\mm)\frac{\|f\|^{1-\frac{1}{p}}_{L^1}}{2^{1-\frac{1}{p}}}\ge W_1(f^+\mm,f^-\mm)
\end{equation}
which is a consequence of the Holder's inequality for the Wasserstein distance (see for instance \cite[Remark 6.6]{Villani} and recall that the measures here have total mass equal to $\|f^{+}\|_{L^1}=\|f^{-}\|_{L^1}=\frac{\|f\|_{L^1}}{2}$).
\end{proof}

\begin{remark}\label{rmk:infty}
We notice that one can recover  an indeterminacy estimate involving the ${\infty}$-Wasserstein distance for example by taking the limit for $p\to +\infty$ in \eqref{ineq:indp} and observing that the constant depending on $p$ does not degenerate for $p\to +\infty$. 
\end{remark}

We conclude the section with an indeterminacy estimate for the Hellinger-Kantorovich distance. In analogy with the comparison between the estimates \eqref{inequality:hellinger-wass} and \eqref{inequality:hellinger-hk}, we obtain an implicit but more refined result than Theorem \ref{theorem:main ind}. Another advantage of the following Theorem is that it is not restricted to functions $f$ with null mean and bounded moment.
\begin{theorem}\label{th: ind hk}
Let $(X,\di,\mm)$ be a metric measure space of finite measure satisfying the $\mathsf{RCD}(K,\infty)$ condition for some $K\in \mathbb{R}$, and let $f\in L^{\infty}(X,\mm)$. Then
\begin{equation}
\hk_{4R_K(t)}(f^+\mm,f^-\mm)\ge \Big(\|f\|_{L^1}-2J_{K}(t)^{\frac{1}{2}}\Per(\{f(x)>0\})^{\frac{1}{2}}\|f\|^{\frac{1}{2}}_{L^{1}}\|f\|^{\frac{1}{2}}_{L^{\infty}}\Big)^{\frac{1}{2}} \quad \forall \ t>0\, ,
\end{equation}
where $R_K(t)$ and $J_K(t)$ were defined in \eqref{def:rk} and \eqref{eq:ExplJK} respectively.
\end{theorem}
\begin{proof}
Using the inequality \eqref{inequality:hellinger-hk} and the definition of $H_t^{\ast}$ \eqref{eq: heatadj} we have that for every $t>0$
\begin{equation}
\hk_{4R_K(t)}^2(f^+\mm,f^-\mm)\ge \hed_{2}^2\big(H_t(f^+)\mm,H_t(f^-)\mm\big).
\end{equation}
With the same estimate as in \eqref{ineq:he2-per ind} we can now bound from below the square of the $2$-Hellinger distance and reach the desired conclusion.
\end{proof}

\section{Proof of the lower bound on the Wassersteind distance of eigenfunctions}\label{sec: eigen}
\begin{proof}[Proof of Theorem \ref{theorem:main eig}]
As in the case of Theorem \ref{theorem:main ind}, we divide the proof in two steps. 

\vspace{3mm}
\emph{Step $1$: general estimate involving time.}

Using the inequality \eqref{inequality:hellinger-wass} with $p=1$ and the definition of $H_t^{\ast}$ \eqref{eq: heatadj} we bound from below the cost $W_1$ in terms of the total variation:
\begin{equation}\label{inequality:wass-eigenfunction}
W_{1}(f_{\lambda}^{+}\mm,f_{\lambda}^{-}\mm)\geq (R_K(t))^{\frac{1}{2}}\hed_{1}(H_t(f_{\lambda}^{+})\mm,H_t(f_{\lambda}^{-})\mm)\, . 
\end{equation}
We observe that
\begin{equation}\label{eq: norm1heat}
\hed_{1}(H_t(f_{\lambda}^{+})\mm,H_t(f_{\lambda}^{-})\mm)=\|H_t(f_{\lambda}^{+})-H_t(f_{\lambda}^{-})\|_{L^{1}(X)}=\|H_t(f_{\lambda})\|_{L^{1}(X)}=e^{-\lambda t}\|f_{\lambda}\|_{L^{1}(X)}, 
\end{equation}
using the linearity of the heat flow and recalling that $H_t(f_{\lambda})=e^{-\lambda t}f_{\lambda}$.

So inequality \eqref{inequality:wass-eigenfunction} reads as 
\begin{equation*}
W_{1}(f_{\lambda}^{+}\mm,f_{\lambda}^{-}\mm)\ge (R_K(t))^{\frac{1}{2}} e^{-\lambda t}\|f_{\lambda}\|_{L^{1}(X)} \quad \forall\, t>0.
\end{equation*}

\vspace{5mm}
\emph{Step $2$: optimizing in $t$.}

In the case $K=0$ the result follows by choosing $\bar t=\frac{1}{2\lambda}$ in the previous inequality. 
\vspace{3mm}

For $K<0$ we choose instead $\bar t= \frac{1}{2K}\log(\frac{\lambda}{\lambda-K})$ in order to obtain 
\begin{equation*}
W_{1}(f_{\lambda}^{+}\mm,f_{\lambda}^{-}\mm)\ge \frac{1}{\sqrt{\lambda}}\sqrt{-\frac{\lambda}{K}\left(e^{(-\frac{\lambda}{K})\log{\frac{\lambda}{\lambda -K}}}-e^{(1-\frac{\lambda}{K})\log{\frac{\lambda}{\lambda -K}}}\right)}\|f_{\lambda}\|_{L^{1}(X)}.
\end{equation*}
The result follows by standard computations, setting $x=-\frac{\lambda}{K}\ge -\frac{M}{K}>0$ and noticing that the function
$$x\mapsto \sqrt{x\left(e^{x\log{\frac{x}{x+1}}}-e^{(1+x)\log{\frac{x}{x+1}}}\right)}=\left(\frac{x}{x+1}\right)^{\frac{x+1}{2}}$$
is increasing. 

\end{proof}
\begin{remark}\label{rmk: no properties}
We notice that in the proof of Theorem \ref{theorem:main eig} we have avoided using fine properties of Laplace eigenfunctions, exploiting only the equality $H_t(f_{\lambda})=e^{-\lambda t}f_{\lambda}$ in the last passage of \eqref{eq: norm1heat}.  
\end{remark}

Along the same lines of Corollary \ref{cor:main ind_p}, one can easily prove the following:
\begin{corollary}\label{cor:main eig}
Let $M>0$, $K\in\Real$ and $(X,\di,\mm)$ be an $\RCD(K,\infty)$ space of finite measure. Then for any non-constant eigenfunction $f_{\lambda}$ of the Laplacian, of eigenvalue  $\lambda\ge M$ and satisfying $\int_{X}\di(\bar x, x)|f_{\lambda}(x)|\,d\mm(x)<+\infty$ for some $\bar x\in X$, it holds for any $p> 1$
	\begin{equation*}
	W_{p}(f_{\lambda}^{+}\mm,f_{\lambda}^{-}\mm)\geq  2^{\frac{p-1}{p}} C(K,M)\frac{1}{\sqrt{\lambda}}\|f_{\lambda}\|_{L^{1}(X)}^{\frac{1}{p}}\, ,
	\end{equation*}
where $C(K,M)$ was defined in \eqref{constant:ck}.
\end{corollary}	

\begin{remark}
	One can recover a lower bound for the $\infty$-Wasserstein distance as in Remark \ref{rmk:infty}.	
\end{remark}

\end{document}